\documentclass{amsart}

\usepackage{amssymb,amsthm,amsmath,verbatim}
\usepackage{t1enc}
\usepackage{enumerate}
\usepackage{color}
\usepackage{tikz}

\usepackage{xmpmulti}
\usepackage{psfrag}

\usepackage[usenames,dvipsnames]{pstricks}
\usepackage{epsfig}
\usepackage{pst-grad}
\usepackage{pstricks,pst-node}
\usepackage{pspicture}
\usepackage{float}
\usepackage[colorlinks]{hyperref}

\numberwithin{equation}{section}

\newtheorem{prop}{Proposition}[section]

\newtheorem{lemma}[prop]{Lemma}

\newtheorem{cor}[prop]{Corollary}
\newtheorem{thm}[prop]{Theorem}

\newtheorem*{mthm}{Main Theorem}
\newtheorem*{mcor}{Corollary}

\newtheorem{clm}[prop]{Claim}

\newtheorem{prob}[prop]{Problem}

\theoremstyle{definition}
\newtheorem{dfn}[prop]{Definition}
\newtheorem{remark}[prop]{Remark}

\DeclareMathOperator{\lex}{lex}

\newcommand{\mb}[1]{\mathbb{#1}}

\newcommand{\oo}{\omega}
\newcommand{\uhr}{\upharpoonright}
\newcommand{\omg}{{\omega_1}}

\def\<{\left\langle}
\def\>{\right\rangle}
\def\br#1;#2;{\bigl[ {#1} \bigr]^ {#2} }

\newcommand{\mbb}[1]{\mathbb{#1}}
\newcommand{\mf}[1]{\mathfrak{#1}}
\newcommand{\ev}[3]{\mathfrak{e}_{#1}(#2,#3)}

\newcommand{\uhp}{\upharpoonright}
\newcommand{\setm}{\setminus}

\newcommand{\supp}{\operatorname{supp}}

\newcommand{\rd}[1]{{h}({#1})}

\begin{document}

\title{Infinite monochromatic sumsets for colourings of the reals}



\author[P. Komj\'ath]{P\'eter Komj\'ath}
\address[P. Komj\'ath]{Institute of Mathematics, E\"otv\"os University
Budapest, P\'azm\'any P. s. 1/C
1117, Hungary}
\email{kope@cs.elte.hu}

\author[I. Leader]{Imre Leader}
\address[I. Leader]{Centre for Mathematical Sciences, Wilberforce Road, Cambridge CB3 0WB, UK}
\email{I.Leader@dpmms.cam.ac.uk}

\author[P. A. Russell]{Paul A. Russell}
\address[P. A. Russell]{Churchill College, Cambridge CB3 0DS, UK}
\email{P.A.Russell@dpmms.cam.ac.uk}

\author[S. Shelah]{Saharon Shelah}
\address[S. Shelah]{Einstein Institute of Mathematics, The Hebrew University of Jerusalem, Jerusalem,
91904, Israel, and Department of Mathematics, Rutgers University, New Brunswick, NJ
08854, USA.}
\email{shelah@math.huji.ac.il}

\author[D. T. Soukup]{D\'aniel T. Soukup}
\address[D.T. Soukup]{Universit\"at Wien,
	Kurt G\"odel Research Center for Mathematical Logic, Austria}
\email[Corresponding author]{daniel.soukup@univie.ac.at}

\author[Z. Vidny\'anszky]{Zolt\'an Vidny\'anszky}
\address[Z. Vidny\'anszky]{Universit\"at Wien,
	Kurt G\"odel Research Center for Mathematical Logic, Austria and Alfr\'ed 	R\'enyi Institute of Mathematics, Hungarian Academy of Sciences, Hungary}
\email{vidnyanszkyz@gmail.com}


\date{\today}

\dedicatory{}

\commby{}

\begin{abstract}
 N. Hindman, I. Leader and D. Strauss proved that it is consistent that there is a finite colouring of $\mb R$ so that no infinite sumset $X+X$ is monochromatic. Our aim in this paper is to prove a
consistency result in the opposite direction: we show that, under certain set-theoretic assumptions, for any finite colouring $c$ of $\mb R$ there is an infinite $X\subseteq \mb R$ so that $c\uhr X+X$ is constant. 
\end{abstract}

\maketitle



	\section{Introduction}
	Neil Hindman's famous sumset theorem states that for any finite colouring of the natural numbers $\mbb N$, there is an infinite set $X$ so that all sums of distinct elements of $X$ are coloured the same. There is a striking difference however, if one allows repetitions in the sumsets; let us begin by recalling a famous open problem of J. Owings:
	
	\begin{prob}[\cite{owings}]\label{owings} Is there a colouring of $\mbb N$ with 2 colours which is not constant on sets of the form $X+X=\{x+y:x,y\in X\}$ whenever $X\subseteq \mbb N$ is infinite?
	\end{prob}
	
	The answer is yes if one is allowed to use 3 colours \cite{hindmanrep}, but, surprisingly, Problem \ref{owings} is still unsolved. Hence, it is very natural to ask, what happens if instead of the natural numbers we colour the real numbers. The following theorem is a recent result of  N. Hindman, I. Leader and D. Strauss:
	
	\begin{thm}[\cite{Hindman}]\label{hindmanrep} It is consistent (namely, it follows from $2^{\aleph_0}=|\mb R|<\aleph_\omega$) that there is a finite colouring $c$ of $\mathbb R$ such that $c$ is not constant on any set of the form $X+X$ where $X\subseteq \mathbb R$ is infinite.
	\end{thm}

	 The group $(\mathbb{R},+)$ is isomorphic to the direct sum of continuum many copies of $(\mathbb{Q},+)$  so we consider colourings of this direct sum. Thus, let $Q(\kappa)$ denote the group $\bigoplus_{\kappa} \mb Q$, i.e., the set of functions with finite support from $\kappa$ to $\mb Q$ with the operation of coordinatewise addition, and let the semigroup $N(\kappa)=\bigoplus_\kappa \mb N$ be defined similarly. As before, if $X$ is a subset of one of the above (semi)groups, let $X+X=\{x+y:x,y\in X\}$, and note that repetitions are allowed in the sumsets. 
	 
	 In order to prove Theorem \ref{hindmanrep}, the authors of \cite{Hindman} showed the following: for each $n\in \omega$, there is a map $c:Q(\aleph_n)\to 2^{4+n}\cdot 9$ such that $c\uhp X+X$ is not constant for any infinite $X\subseteq Q(\aleph_n)$.

	\medskip

	It was asked in \cite[Question 2.9]{Hindman} if the conclusion of Theorem \ref{hindmanrep} is true in ZFC. Note that given an infinite semigroup $(G,+)$ and finite colouring $c:G\to r$ one can always find an infinite set $X\subseteq G$ so that both $\{x+y:x\neq y\in X\}$ and $\{x+x:x\in X\}$ are monochromatic. This is by the Ramsey theorem and the pigeon hole principle, respectively. However, the constant values of $c$ on these two sets might differ. Arranging that these two colours are the same is the main difficulty in proving a positive partition result.

	The purpose of this paper is to show that the conclusion of Theorem \ref{hindmanrep} can fail in some models of ZFC constructed with the use of a large cardinal.
	
	\begin{mthm}
		Consistently relative to an $\omg$-Erd\H os cardinal, for any $c:N(2^{\aleph_0})\to r$ with $r$ finite there is an infinite $X\subseteq N(2^{\aleph_0})$ so that $c\uhr X+X$ is constant.
	\end{mthm}

	The proof of the Main Theorem is a combination of various ideas from the six authors from between 2015 and 2017. I.  Leader and P.  A. Russell \cite{leaderrussell}, and independently P. Komj\'ath\footnote{Personal communication.} proved that if a colouring $c:N(\mf \kappa)\to r$ is \emph{canonical} in some sense\footnote{That is, $c(x)$ only depends on the values of $x$ in its 'support' but not the support itself.} on a large set then infinite monochromatic sumsets can be found. Hence, the conclusion of the Main Theorem was known for $N(\mf \kappa)$ instead of $N(2^{\aleph_0})$ where $\kappa$ is large enough relative to the number of colours $r$. For example, $\kappa\geq \beth_\oo$ suffices for any finite $r$, since it allows the application of the Erd\H os-Rado partition theorem with high exponents; the interested reader can consult \cite{leaderrussell} or see Lemma \ref{lm:leader} below. 
	
	Since $2^{\aleph_0}$ badly fails such strong positive partition relations, other ideas were required. D. Soukup and Z. Vidny\'anszky analysed the situation further to see what alternative partition relation might suffice to carry out some form of such a canonization argument.  It was S. Shelah who suggested using \cite{wsr2} and indeed, the last two authors found a way to combine \cite[Theorem 3.1]{wsr2} with all the aforementioned machinery resulting in the Main Theorem.
	\medskip

	Since $N(2^{\aleph_0})$ embeds into $\mb R$, we immediately have the desired conclusion.
	
	\begin{mcor}
		Consistently relative to an $\omg$-Erd\H os cardinal, for any finite colouring $c$ of $\mb R$  there is an infinite $X\subseteq \mb R$ so that $c\uhr X+X$ is constant. 
	\end{mcor}

\begin{remark}
We recently learned from Jing Zhang\footnote{Personal communication.} that he showed that the conclusion of our main theorem holds if one adds $\aleph_\omega$ many Cohen reals to a model of GCH.\footnote{I.e., the Generalized Continuum Hypothesis which says that $2^\kappa=\kappa^+$ for any infinite cardinal $\kappa$.} In particular, no large cardinals are necessary and the continuum can be as small as possible (cf. Theorem \ref{hindmanrep}). We conclude our paper with further open problems.
\end{remark}

	\medskip

	Let us briefly mention the related problem of studying \emph{uncountable} sumsets. Starting with \cite{Hindman} a sequence of papers \cite{kope, soukupweiss, david, davidassaf, kope2, carlucci} elaborates on this question and its relatives; we refer the reader to the introduction of \cite{davidassaf} for an excellent summary.  In particular, A. Rinot and D. Fern\'andez-Bret\'on showed that one can colour $\mb R$ with 2 colours ($\oo$ many colours) such that all colours appear on all sumsets of size $\aleph_1$ (sumsets of size $2^{\aleph_0}$, respectively); both of these results are provable in ZFC \cite{davidassaf} and hold even if repetitions are not allowed in the sumsets. This implies that our Main Theorem is optimal in the sense that we cannot hope to find uncountable monochromatic sumsets.
	
	\medskip


	
	Throughout the paper, we will use standard notations and facts which can be found in \cite{jech}. Our proof is elementary; it only uses independence results as black boxes, so it can be understood without any familiarity with the techniques of forcing, large cardinals etc.

	\subsection*{Further acknowledgements} We discussed the question at hand and received helpful comments from several other people, including M. Elekes, A. Rinot, L. Soukup, S. Todorcevic, and W. Weiss. We are grateful for the anonymous referee for his/her careful reading and valuable advice which made our exposition more accessible.
	
	\smallskip
	
	D. Soukup was supported in part by PIMS, the	National Research, Development and Innovation Office--NKFIH grant no. 113047 and the FWF Grant I1921. This research was partially done whilst  visiting  the Isaac Newton Institute for Mathematical Sciences part of the programme `Mathematical, Foundational and Computational Aspects of the Higher Infinite' (HIF) funded by EPSRC grant EP/K032208/1.
	
	\smallskip
	Z. Vidny\'anszky was partially supported by the	National Research, Development and Innovation Office--NKFIH grants no.~113047, no.~104178 and no.~124749 and by FWF Grant P29999.

	\section{The proof of the Main Theorem}
	
	The rest of our paper is devoted to proving our Main Theorem. We start by recalling the main ideas of \cite{leaderrussell} in Section \ref{sec:leader} which will motivate both the somewhat technical canonization result of Shelah that we state in Section \ref{sec:sh} and our proof of the Main Theorem in Section \ref{sec:proof}

	\subsection{Patterns and homogeneous sets}\label{sec:leader} Let $\kappa$ be a cardinal.
	Given $s\in \mb N^{<\omega}$ and $a\in [\kappa]^{|s|}$, we define $x=s*a\in N(\kappa)$ by $\supp(x)=a$ and $x(a(i))=s(i)$ where $\{a(i):i<|s|\}$ is the increasing enumeration of $a$. 
	
	Now, if $c:N(\kappa)\to r$ is a colouring and $s\in \mb N^{<\omega}$, we define $c_s:[\kappa]^{|s|}\to r$ by $$c_s(a)=c(s*a).$$ 
	
	That is, any colouring $c$ of the direct sum $N(\kappa)$ and each pattern $s$ naturally defines a colouring $c_s$ of finite subsets of the continuum.
	The main idea from  \cite{leaderrussell} is the following: given an $r$-colouring $c$ of $N(\kappa)$, there are $r+1$ many patterns $(s_l)_{l\leq r}$ so that if each colouring $c_{s_l}$ is constant on a large set $W\subseteq \kappa$ then we can find an infinite $X\subseteq N(\kappa)$ so that $c\uhr X+X$ is constant.
	
	Let us define these patterns: we let $s_l=(2,2, \dots ,2,2,4,\dots ,4)$ where $2l$-many 2 are followed by $(r-l)$-many 4 for each $l\leq r$; see Figure \ref{fig:patternn}. So, for a fixed $r$, we defined $(r+1)$-many patterns.
		
		\begin{figure}[H]
			\psscalebox{0.7 0.7} 
			{
				\begin{pspicture}(0,-2.115)(8.898557,2.115)
				\psline[linecolor=black, linewidth=0.04](0.8,0.285)(0.8,-1.315)
				\psline[linecolor=black, linewidth=0.04](1.6,0.285)(1.6,-1.315)
				\psdots[linecolor=black, dotsize=0.2](1.6,0.285)
				\psdots[linecolor=black, dotsize=0.2](0.8,0.285)
				\psline[linecolor=black, linewidth=0.04](4.4,0.285)(4.4,-1.315)
				\psline[linecolor=black, linewidth=0.04](5.2,0.285)(5.2,-1.315)
				\psdots[linecolor=black, dotsize=0.2](5.2,0.285)
				\psdots[linecolor=black, dotsize=0.2](4.4,0.285)
				\psline[linecolor=black, linewidth=0.04](6.0,1.885)(6.0,-1.315)
				\psdots[linecolor=black, dotsize=0.2](6.0,1.885)
				\psline[linecolor=black, linewidth=0.04](6.8,1.885)(6.8,-1.315)
				\psdots[linecolor=black, dotsize=0.2](6.8,1.885)
				\psline[linecolor=black, linewidth=0.04](8.8,1.885)(8.8,-1.315)
				\psdots[linecolor=black, dotsize=0.2](8.8,1.885)
				\psdots[linecolor=black, dotsize=0.1](2.4,-0.515)
				\psdots[linecolor=black, dotsize=0.1](2.8,-0.515)
				\psdots[linecolor=black, dotsize=0.1](3.2,-0.515)
				\psdots[linecolor=black, dotsize=0.1](7.2,-0.515)
				\psdots[linecolor=black, dotsize=0.1](7.6,-0.515)
				\psdots[linecolor=black, dotsize=0.1](8.0,-0.515)
				\rput[bl](2.4,-2.115){$2l$-many}
				\rput[bl](6.4,-2.115){$(r-l)$-many}
				\rput[bl](0.0,0.285){2}
				\rput[bl](5.2,1.885){4}
				\end{pspicture}
			}
			\caption{The pattern $s_l$.}
			\label{fig:patternn}
		\end{figure}
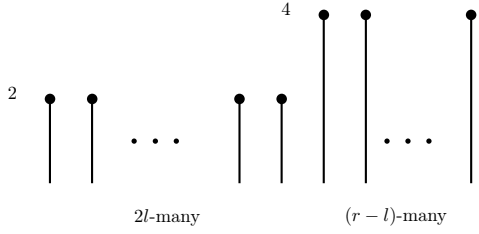

		Let us consider the special case of $r=2$ when  $s_0=(4,4), s_1=(2,2,4)$ and $s_2=(2,2,2,2)$. Suppose that $W\subseteq \kappa$ has order type $\omega+1$ and all three colourings $c_{s_0}$, $c_{s_1}$ and $c_{s_2}$ are constant on $W$. Now, two of these constant values must agree, say  both $c_{s_0}$ and $c_{s_1}$ are constant 0. Pick ordinals $\alpha_0<\alpha_1<\dots<\alpha_\omega$ in $W$ and let $x_i=\{(\alpha_i,2),(\alpha_\omega,2)\}$ for $i<\oo$. It is easy to see that $$c(2x_i)=c_{s_0}(\alpha_i,\alpha_\omega)=0=c_{s_1}(\alpha_i,\alpha_j,\alpha_\omega)=c(x_i+x_j)$$ for any $i<j<\omega$. So $c\uhr X+X=0$ for $X=\{x_i:i<\oo\}$.
		
		\medskip
		
		In general, the following holds:
		
		\begin{lemma}\label{lm:leader} Suppose that $c:N(\kappa)\to r$ and there is some $W\subseteq \kappa$ of order type $\omega+r-1$ so that each $c_{s_l}\uhr [W]^{r+l}$ is constant. Then there is an infinite $X\subseteq N(\kappa)$ so that $c\uhr X+X$ is constant. 
		\end{lemma}
	\begin{proof} Suppose that $W=\{\alpha_i:i<\omega+r-1\}$ satisfies the above requirements. Now, pick $l<k\leq r$ so that $c_{s_l}$ and $c_{s_k}$ have the same constant value on $W$. We define $x_i\in N(\kappa)$ for $i<\oo$  as follows: the support of $x_i$ has the form $a\cup b_i\cup c$ where $a=\{\alpha_0,\alpha_1,\dots ,\alpha_{2l-1}\}$, $c=\{\alpha_\omega,\dots,\alpha_{\omega+r-k-1}\}$ and the sets $\{b_i:i<\oo\}$ are pairwise disjoint of size $k-l$ selected from $\{\alpha_n:2l\leq n<\oo\}$. We define $x_i\uhr a$ constant 1 and $x_i\uhr (b_i\cup  c)$ constant 2. It is easy to check that $$c(2x_i)=c_{s_l}(a\cup b_i\cup c)=c_{s_k}(a\cup b_i\cup b_j\cup c)=c(x_i+x_j)$$ for any $i<j<\omega$. So $c\uhr X+X$ is constant for $X=\{x_i:i<\oo\}$.
		\end{proof}
		
		Almost the same result as Lemma \ref{lm:leader} appears in the proof of \cite[Theorem 2]{leaderrussell} however the authors require a slightly larger order-type for $W$. We also learned from J. Zhang that he used the $r=2$ case of Lemma \ref{lm:leader} to note that for any 2-colouring of $\mb R$, one can find an infinite $X$ so that $X+X$ is monochromatic.
	
However, it is well known that for $\kappa=2^{\aleph_0}$ we cannot necessarily find the $W$ that is required to apply the above result for $r>2$. Still, we can (consistently) find a large subset $W$ and a map $F:W \to 2^\omega$ such that the colourings $c_{s_l}$ are simple on $W$ in the sense that the colour of a tuple is essentially determined by the equivalence class of its $F$ image with respect to a natural equivalence relation on finite tuples on $2^\omega$. Making the latter precise is the content of the next section.

	\subsection{Shelah's canonization theorem}\label{sec:sh} 
	Our goal now is to explain \cite[Theorem 3.1 (2)]{wsr2} and we aimed to follow the definitions there.  As usual, for distinct $s,t \in 2^\omega$ let  $\Delta(s,t)$ stand for $\min\{n:s(n)\not =t(n)\}$. 
	
	We now define the equivalence relation that we mentioned earlier.
	

	
	\begin{dfn}
    \label{d:similar}
	We say that $\bar t,\bar s\in (2^\oo)^k$ are \emph{similar} iff for all $l_1,l_2,l_3,l_4<k$:
	\begin{enumerate}
		\item $\Delta(\bar t(l_1),\bar t(l_2))<\Delta(\bar t(l_3),\bar t(l_4))$ iff  $\Delta(\bar s(l_1),\bar s(l_2))<\Delta(\bar s(l_3),\bar s(l_4)) ,$
		\item  $$\bar t(l_3)\uhr n<_{\lex} \bar t(l_4)\uhr n \textmd{ for }n=\Delta(\bar t(l_1),\bar t(l_2))$$ iff $$\bar s(l_3)\uhr m<_{\lex} \bar s(l_4)\uhr m \textmd{ for }m=\Delta(\bar s(l_1),\bar s(l_2)),$$
		
		\item $$\bar t(l_3)(n)=0 \textmd{ for }n=\Delta(\bar t(l_1),\bar t(l_2))$$ iff $$\bar s(l_3)(m)=0 \textmd{ for }m=\Delta(\bar s(l_1),\bar s(l_2)).$$
	\end{enumerate}
		\end{dfn}
		
		The above definition says that the ordered tuples $\bar t$ and $\bar s$ have  the same branching pattern (conditions (1) and (2)) and the same values at the corresponding levels (condition (3)). Observe that there are only finitely many equivalence classes for similarity in  $(2^\oo)^k$.
		
		\medskip

		Now, given $d:[W]^k \to r$ and $F:W \to 2^\oo$, we say that  
		\begin{enumerate}[(a)]
		    \item two tuples $\bar \alpha=\{\alpha_0<\dots<\alpha_{k-1}\},\bar \beta=\{\beta_0<\dots<\beta_{k-1}\}\subset W$ are \emph{$F$-similar} if $(F(\alpha_l))_{l<k}$ and $(F(\beta_l))_{l<k}$ are similar, and
		    \smallskip
		    \item 	$d$ is \emph{$F$-canonical} if $d(\bar \alpha)=d(\bar \beta)$ for any two $F$-similar tuples $\bar \alpha,\bar \beta$ from $W$. 
		\end{enumerate}
	
		\medskip

	We are ready to state Shelah's theorem.

	\begin{thm}\cite[Theorem 3.1 (2)]{wsr2}\label{Shmodel}
		Suppose that $\lambda$ is an $\omg$-Erd\H os cardinal in $V$. Then there is a forcing notion $\mb P$ so that $V^\mb P$ satisfies the following:
		\begin{enumerate}
			\item $ 2^{\aleph_0}=\lambda$,
			\item $\textmd{MA}_{\aleph_1}(\textmd{Knaster})$ (see below), and
			\item  for any $d:[\lambda]^k\to r$ (with $r,k\in \oo$) there is $W\in [\lambda]^{\aleph_1}$ and injective $F:W\to 2^\oo$ so that $d\uhr [W]^k$ is $F$-canonical. 
		\end{enumerate}
	\end{thm}

	In particular, the above theorem says that no matter how large the number of colours $r$ is, the restricted colouring $d\uhr [W]^k$ assumes at most as many colours as the number of equivalence classes for similarity. For $k=2$, this gives at most two colours.
	
	We also mention that the forcing $\mb P$ is of the form $\mb P_0*\mb P_1$ where $\mb P_0$ is $<\lambda$-closed  of size $\lambda$ and $\mb P_1\in V^{\mb P_0}$ is ccc. $\mb P$ collapses no cardinals $\leq \lambda$ so the continuum will be very large (i.e., a former large cardinal). In fact, we do not know if Theorem \ref{Shmodel} could hold for say $2^{\aleph_0}=\aleph_{\oo+1}$ or if the use of large cardinals can be avoided in this result.
	
Let us omit the  definition of $\textmd{MA}_{\aleph_1}(\textmd{Knaster})$ as we only need a particular corollary of this axiom.\footnote{The Knaster property and Martin's axiom are covered by \cite{jech} in detail.} 
	
	\begin{thm}
		Suppose $\textmd{MA}_{\aleph_1}(\textmd{Knaster})$. Then for any $r<\oo$ and $g:\oo\times \omg\to r$  there are $A\in [\oo]^\oo$, $B\in [\omg]^{\omg}$so that $g\uhr A\times B$ is constant.
	\end{thm}
	
	 The conclusion above is abbreviated as ${{\omg}\choose{\oo}}\to {{\omg}\choose{\oo}}^{1,1}_r$. For the interested reader, we mention that  $\textmd{MA}_{\aleph_1}(\textmd{Knaster})$ implies $\aleph_1<\mathfrak s$ by \cite[Theorem 7.7]{blass} and the discussion there, and $\aleph_1<\mathfrak s$ implies ${{\omg}\choose{\oo}}\to {{\omg}\choose{\oo}}^{1,1}_r$ by \cite[Claim 2.4]{garti}. The only important thing for us is that ${{\omg}\choose{\oo}}\to {{\omg}\choose{\oo}}^{1,1}_r$ holds in our model.

	\subsection{Proving the Main Theorem}\label{sec:proof}We prove that the conclusion of our Main Theorem holds in any model that satisfies conditions (1)-(3) of Theorem \ref{Shmodel}. So, for the rest of the paper, let us fix a colouring $c:N(2^{\aleph_0})\to r$ with some $r<\oo$.
	
	We use the patterns $s_0,\dots,s_r$ as defined in Section \ref{sec:leader}. Note that the $r+1$ many colourings $c_{s_l}:[2^{\aleph_0}]^{r+l}\to r$ can be coded by a single map $d:[2^{\aleph_0}]^{2r}\to r^r$. In turn, by applying Theorem \ref{Shmodel},  there is some $W\in [2^{\aleph_0}]^{\aleph_1}$ and injective $F:W\to 2^\oo$ so that $c_{s_l}\uhr [W]^{r+l}$ is $F$-canonical for all $l\leq r$. 
	
	\medskip
	
	Our goal now is to define infinite subsets $A_0,\dots,A_{r-1}\subseteq W$ and a family of $r+\ell$-tuples from elements of $\bigcup_{i< r} A_i$ for each $\ell\leq r$ (which we will call \emph{canonical $\ell$-candidates}) so that any two of them are $F$-similar. Steps 1-4 below detail  the definition of canonical $\ell$-candidates  and the construction of the sets $A_l$ which involve controlling the branching pattern and values at branching levels. 
	This will allow us in Step 5 to carry out a similar argument to the proof of Lemma \ref{lm:leader} and find an infinite $X$ with $c\uhr X+X$ constant. 
	
		\medskip
	
\textbf{Step 1.}	We start by selecting $A_0, \dots ,A_{r-1}\subset W$ so that \begin{enumerate}
			\item $F''A_l<F''A_k$ (in $2^\oo$) for all $l<k< r$,
			\item $|A_l|=\aleph_1$ for $l<r$, and
			\item there is an increasing sequence $\nu_l<\oo$ (for $l<r-1$) so that
			\begin{itemize}
				\item $\Delta_F(\alpha,\alpha')=\nu_l$ for all $\alpha\in A_l,\alpha\in A_k$ with $l<k<r$, and
				\item $\Delta_F(\alpha,\alpha')>\nu_{r-2}$ for all $\alpha,\alpha'\in A_l$ with $l<r$.
				
			\end{itemize}
			%
		\end{enumerate}
		
		\begin{figure}[H]
			\psscalebox{0.7 0.7} 
			{
				\begin{pspicture}(0,-3.31)(12.2,3.31)
				\psellipse[linecolor=black, linewidth=0.03, dimen=outer](1.5,2.13)(1.5,0.4)
				\psellipse[linecolor=black, linewidth=0.03, dimen=outer](5.0,2.11)(1.5,0.4)
				\psellipse[linecolor=black, linewidth=0.03, dimen=outer](10.7,2.13)(1.5,0.4)
				\psline[linecolor=black, linewidth=0.03](2.0,-3.07)(10.56,-1.67)
				\psline[linecolor=black, linewidth=0.03](3.24,-2.85)(1.66,-1.21)
				\psline[linecolor=black, linewidth=0.03](6.52,-2.33)(5.2,-0.69)
				\psline[linecolor=black, linewidth=0.03](8.32,-0.39)(9.32,-1.85)
				\psline[linecolor=black, linewidth=0.03, linestyle=dashed, dash=0.17638889cm 0.10583334cm](0.6,1.53)(1.64,-1.21)(2.4,1.53)
				\psline[linecolor=black, linewidth=0.03, linestyle=dashed, dash=0.17638889cm 0.10583334cm](9.8,1.53)(10.6,-0.43)(11.4,1.53)
				\psdots[linecolor=black, dotsize=0.1](7.2,2.13)
				\psdots[linecolor=black, dotsize=0.1](7.6,2.13)
				\psdots[linecolor=black, dotsize=0.1](8.0,2.13)
				\rput[bl](0.88,3.01){$F''A_0$}
				\rput[bl](4.2,2.93){$F''A_1$}
				\rput[bl](10.2,2.93){$F''A_{r-1}$}
				\rput[bl](3.9,-3.31){$\nu_0$}
				\rput[bl](7.18,-2.73){$\nu_1$}
				\rput[bl](9.66,-2.27){$\nu_{r-2}$}
				\psline[linecolor=black, linewidth=0.03, linestyle=dashed, dash=0.17638889cm 0.10583334cm](4.38,1.53)(5.18,-0.69)(5.72,1.57)
				\psline[linecolor=black, linewidth=0.03](10.56,-1.69)(10.6,-0.43)
				\end{pspicture}
			}
			\caption{The $r$ uncountable blocks from $W$ with nicely ordered images on $2^\omega$.}
			\label{fig:blocks}
			%
			%
			%
			
		\end{figure}
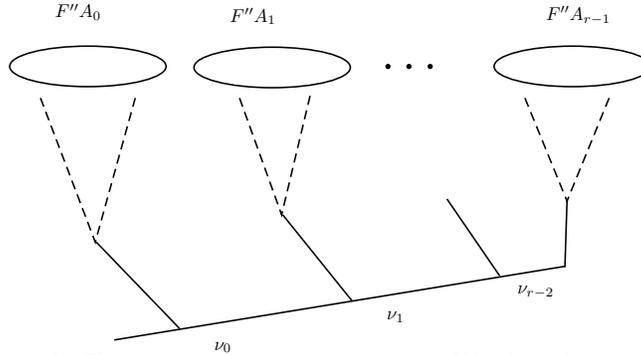

		See Figure \ref{fig:blocks}. Now, this ensures that all tuples $(\alpha_0,\dots ,\alpha_{r-1})_<\in A_0\times \dots \times A_{r-1}$ belong to the same similarity type. Here,  we use the notation $\{\alpha_i\}_<$ or $(\alpha_i)_<$ to denote $\alpha_0<\alpha_1<\dots$, in other words, the fact that the sequence of $\alpha_i$'s is increasing in the ordering of ordinals.  

\begin{dfn}	 An \emph{$\ell$-candidate with respect to $\textbf A=(A_l)_{l< r}$} is an $(r+\ell)$-tuple of elements of $W$ where we take two elements from each of $A_0,\dots, A_{\ell-1}$ and a single point from each $A_k$ for $\ell\leq k< r$.
		\end{dfn}
		
		In particular, a $0$-candidate selects a single element from each $A_l$. Observe that the choice of the $A_l$'s ensures that all $0$-candidates  are $F$-similar and so all elements $s_0*\bar \alpha$ are assigned the same colour by $c$. 
		
\medskip

\textbf{Step 2.}		In general, not all $\ell$-candidates are necessarily similar.
		In particular, we need to look at the new splitting levels $\delta$ (coming from the first $\ell$-many pairs) and make sure that the same values appear when we evaluate other branches at $\delta$. We need to evaluate at branches lying above and below where the splitting occurs. We deal with the former first and the latter in Claim \ref{clm:thin2}. 
		
		\begin{clm}\label{clm:thin1}
			There is $A_l'=\{\alpha_i^l:i\leq\oo\}_<\subseteq A_l$ for each $l< r$ so that 
			\begin{enumerate}
				\item\label{it:op} $F\uhr \bigcup_{l<r}A_l'$ is order preserving and $\sup_{i\in \oo}F(\alpha^l_i)= F(\alpha^l_\oo)$ in $2^\oo$,
				\item\label{it:spl} there are strictly increasing $(\delta^l_i)_{i<\oo}$ so that \begin{itemize}
					\item[(a)] $\Delta_F(\alpha^l_i,\alpha^l_j)=\delta^l_i$ for all $i<j\leq \oo$, and
					\item[(b)]  $\delta^l_i<\delta^k_i<\delta^l_{i+1}$ for $l<k< r$ and $i<\oo$,
				\end{itemize}
				
				\item\label{it:eval} there are $u_l: \{l+1,\dots,r-1\}\to 2$ for $l<r$ such that $u_l(k)=F(\alpha)(\delta^l_i)$ for any $\alpha \in A_k$ and $l<k<r$.
				
			\end{enumerate}
			
		\end{clm}
		
		See Figure \ref{fig:thin1} for a picture; we have marked the important splitting levels on later branches, with the same symbols marking the same values.

		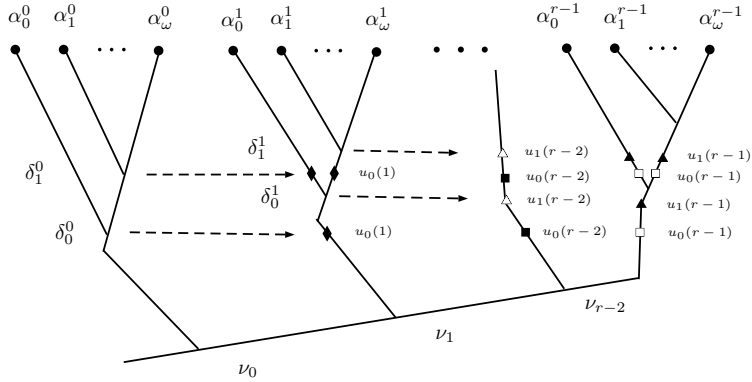
\begin{figure}[H]
			\psscalebox{0.8 0.8} 
			{
				\begin{pspicture}(0,-3.115)(12.32,3.115)
				\psline[linecolor=black, linewidth=0.03](1.92,-2.875)(10.48,-1.475)
				\psline[linecolor=black, linewidth=0.03](3.16,-2.655)(1.58,-1.015)
				\psline[linecolor=black, linewidth=0.03](6.44,-2.135)(5.12,-0.495)
				\psline[linecolor=black, linewidth=0.03](8.24,-0.195)(9.24,-1.655)
				\psdots[linecolor=black, dotsize=0.1](7.12,2.325)
				\psdots[linecolor=black, dotsize=0.1](7.52,2.325)
				\psdots[linecolor=black, dotsize=0.1](7.92,2.325)
				\rput[bl](3.82,-3.115){$\nu_0$}
				\rput[bl](7.1,-2.535){$\nu_1$}
				\rput[bl](9.58,-2.075){$\nu_{r-2}$}
				\psline[linecolor=black, linewidth=0.03](10.48,-1.495)(10.52,-0.235)
				\psdots[linecolor=black, dotsize=0.16](4.54,2.305)
				\psdots[linecolor=black, dotsize=0.16](3.74,2.305)
				\psdots[linecolor=black, dotsize=0.16](6.12,2.285)
				\psdots[linecolor=black, dotsize=0.06](5.12,2.305)
				\psdots[linecolor=black, dotsize=0.06](5.32,2.305)
				\psdots[linecolor=black, dotsize=0.06](5.52,2.305)
				\psline[linecolor=black, linewidth=0.03](5.14,-0.535)(6.1,2.285)
				\psline[linecolor=black, linewidth=0.03, linestyle=dashed, dash=0.17638889cm 0.10583334cm, arrowsize=0.05291667cm 2.0,arrowlength=1.4,arrowinset=0.0]{->}(2.02,-0.715)(4.76,-0.755)
				\psline[linecolor=black, linewidth=0.03](3.76,2.285)(5.28,-0.115)
				\psline[linecolor=black, linewidth=0.03, linestyle=dashed, dash=0.17638889cm 0.10583334cm, arrowsize=0.05291667cm 2.0,arrowlength=1.4,arrowinset=0.0]{->}(2.3,0.245)(4.76,0.245)
				\psline[linecolor=black, linewidth=0.03](4.58,2.265)(5.54,0.625)
				\psline[linecolor=black, linewidth=0.03](0.96,2.285)(1.92,0.245)
				\psdots[linecolor=black, dotsize=0.06](1.52,2.325)
				\psdots[linecolor=black, dotsize=0.06](1.72,2.325)
				\psdots[linecolor=black, dotsize=0.06](1.92,2.325)
				\psdots[linecolor=black, dotsize=0.16](2.5,2.305)
				\psdots[linecolor=black, dotsize=0.16](0.12,2.325)
				\psdots[linecolor=black, dotsize=0.16](0.92,2.325)
				\psline[linecolor=black, linewidth=0.03](0.12,2.305)(1.64,-0.755)
				\psline[linecolor=black, linewidth=0.03](1.58,-1.035)(2.52,2.325)
				\psline[linecolor=black, linewidth=0.03](10.52,-0.255)(11.68,2.345)
				\psline[linecolor=black, linewidth=0.03](9.28,2.325)(10.64,0.005)
				\psdots[linecolor=black, dotsize=0.16](10.08,2.345)
				\psdots[linecolor=black, dotsize=0.16](9.28,2.345)
				\psdots[linecolor=black, dotsize=0.16](11.66,2.325)
				\psdots[linecolor=black, dotsize=0.06](10.68,2.345)
				\psdots[linecolor=black, dotsize=0.06](10.88,2.345)
				\psdots[linecolor=black, dotsize=0.06](11.08,2.345)
				\psline[linecolor=black, linewidth=0.03](10.12,2.305)(11.1,1.105)
				\psline[linecolor=black, linewidth=0.03, linestyle=dashed, dash=0.17638889cm 0.10583334cm, arrowsize=0.05291667cm 2.0,arrowlength=1.4,arrowinset=0.0]{->}(5.5,-0.115)(7.64,-0.155)
				\psline[linecolor=black, linewidth=0.03, linestyle=dashed, dash=0.17638889cm 0.10583334cm, arrowsize=0.05291667cm 2.0,arrowlength=1.4,arrowinset=0.0]{->}(5.74,0.645)(7.48,0.605)
				\psline[linecolor=black, linewidth=0.03](8.26,-0.215)(8.1,1.985)
				\psdots[linecolor=black, dotstyle=diamond*, dotsize=0.18](5.3,-0.735)
				\psdots[linecolor=black, dotstyle=square*, dotsize=0.16](8.6,-0.715)
				\psdots[linecolor=black, dotstyle=square, dotsize=0.16, fillcolor=white](10.5,-0.715)
				\psdots[linecolor=black, dotstyle=diamond*, dotsize=0.18](5.04,0.265)
				\psdots[linecolor=black, dotstyle=diamond*, dotsize=0.18](5.42,0.265)
				\psdots[linecolor=black, dotstyle=square*, dotsize=0.16](8.26,0.185)
				\psdots[linecolor=black, dotstyle=triangle, dotsize=0.16, fillcolor=white](8.22,0.585)
				\psdots[linecolor=black, dotstyle=triangle*, dotsize=0.16](10.88,0.505)
				\psdots[linecolor=black, dotstyle=triangle*, dotsize=0.16](10.32,0.525)
				\psdots[linecolor=black, dotstyle=triangle, dotsize=0.16, fillcolor=white](8.28,-0.195)
				\psdots[linecolor=black, dotstyle=triangle*, dotsize=0.16](10.52,-0.255)
				\rput[bl](5.76,-0.815){\tiny{$u_0(1)$}}
				\rput[bl](10.96,-0.835){\tiny{$u_0(r-1)$}}
				\rput[bl](8.88,-0.815){\tiny{$u_0(r-2)$}}
				\rput[bl](5.82,0.165){\tiny{$u_0(1)$}}
				\rput[bl](8.6,0.105){\tiny{$u_0(r-2)$}}
				\rput[bl](8.62,-0.275){\tiny{$u_1(r-2)$}}
				\rput[bl](8.58,0.525){\tiny{$u_1(r-2)$}}
				\rput[bl](10.94,-0.355){\tiny{$u_1(r-1)$}}
				\rput[bl](11.11,0.105){\tiny{$u_0(r-1)$}}
				\rput[bl](11.28,0.445){\tiny{$u_1(r-1)$}}
				\rput[bl](0.0,2.685){$\alpha^0_0$}
				\rput[bl](0.76,2.705){$\alpha^0_1$}
				\rput[bl](2.32,2.665){$\alpha^0_\omega$}
				\rput[bl](4.36,2.665){$\alpha^1_1$}
				\rput[bl](3.54,2.625){$\alpha^1_0$}
				\rput[bl](5.94,2.645){$\alpha^1_\omega$}
				\rput[bl](9.9,2.665){$\alpha^{r-1}_1$}
				\rput[bl](8.78,2.685){$\alpha^{r-1}_0$}
				\rput[bl](11.48,2.665){$\alpha^{r-1}_\omega$}
				\psdots[linecolor=black, dotstyle=square, dotsize=0.16, fillcolor=white ](10.48,0.265)
				\psdots[linecolor=black, dotstyle=square, dotsize=0.16, fillcolor=white](10.76,0.265)
				\rput[bl](0.78,-0.915){$\delta^0_0$}
				\rput[bl](0.28,0.085){$\delta^0_1$}
				\rput[bl](4.2,-0.315){$\delta^1_0$}
				\rput[bl](3.98,0.505){$\delta^1_1$}
				\end{pspicture}
			}
			\caption{Selecting subsequences with the important positions marked.}
			\label{fig:thin1}
		\end{figure}

		\begin{proof}
			First, shrink $A_0$ to some $A''_0$ of type $\oo+1$ so that $F\uhr A''_0$ is increasing, continuous, and  (\ref{it:spl})(a) is satisfied; next, replace $A_1,\dots,A_{r-1}$ by their uncountable subsets so that every element of their union is greater than every element of $A''_0$ (we use the notation $A_1,\dots,A_{r-1}$ for the shrinked sets as well). Note that (\ref{it:spl})(a) will still hold no matter how we shrink $A_0''$ further (just the splitting sequence is redefined).
			
			Define $$g_{0k}: (A_0''\setm \{\max A''_0\})\times A_k\to 2$$ for $0<k<r$ by $g_{0k}(\alpha^0_i,\beta)=F(\beta)(\delta^0_i)\in 2$. Now, apply  ${{\omg}\choose{\oo}}\to {{\omg}\choose{\oo}}^{1,1}_2$ to each $g_{01},g_{02},\dots ,g_{0r-1}$, successively  shrinking $A_0''\setm \{\max A_0''\}$ $(r-1)$-many times and each $A_k$ once. The constant values $u_0(k)$ for $g_k$ will define  $u_0:{r\setm 1}\to 2$. Note that we  made sure that (\ref{it:eval}) is satisfied no matter how we shrink $A''_0$ or $A_l$ for $0<l<r-1$ in later steps.
			
			We move onto $A_1$ and proceed similarly: first, shrink $A_1$ to a type $\oo+1$ sequence so that the restriction of $F$ is increasing, continuous, and  (\ref{it:spl})(a) is satisfied. Then shrink the sets $A_2,\dots,A_{r-1}$ so that every element of their union is greater than the elements of $A_1$. Next, shrink each set further so that (\ref{it:eval}) holds for $l=1$ using ${{\omg}\choose{\oo}}\to {{\omg}\choose{\oo}}^{1,1}_2$; this defines $u_1$. Then we move to $A_2$ etc.
			
			Notice that this process yields $A_l'\subseteq A_l$ so that  (\ref{it:op}), (\ref{it:spl})(a) and (\ref{it:eval}) holds. Finally, we do a final (simultaneous) shrinking to ensure (\ref{it:spl})(b).
			
		\end{proof}
		
		Again, to make notation lighter, we drop the primes i.e., we forget about the original uncountable $A_l$ defined in Step 2, and call our new $\oo+1$-sequences from Claim \ref{clm:thin1} $A_l$ (instead of $A_l'$).
		
		\medskip
		
		Now, what decides the similarity type of an $\ell$-candidate with respect to this new sequence $\textbf A=(A_l)_{l< r}$? Firstly, we will focus only on $\ell$-candidates of a certain specific form:
			\begin{dfn}
				A \emph{canonical $\ell$-candidate} is an $\ell$-candidate with respect to $\textbf A$ of the form $$\bar \alpha =(\alpha^0_{i_0},\alpha^0_{j_0},\dots, \alpha^{\ell-1}_{i_{\ell-1}},\alpha^{\ell-1}_{j_{\ell-1}},\alpha^\ell_{i_\ell},\dots, \alpha^{r-1}_{i_{r-1}})$$ so that \begin{enumerate}
					\item  $i_k<j_k\leq \oo$ for $k<\ell$ and $i_0\leq i_1\leq \dots \leq i_{r-1}\leq \oo$, and
					\item if $j_k\neq \oo$ then $\delta^k_{j_k}>\delta^{\ell-1}_{i_{\ell-1}}$.
				\end{enumerate}
		\end{dfn}
		
		For example, if we fix $\ell<k<r$ and $i<\oo$ then  $$\{\alpha^l_0,\alpha^{l}_\oo:l<\ell\}\cup\{\alpha^l_{i}:\ell\leq l<k\}\cup \{\alpha^l_\oo:k\leq l<r\}$$ is a canonical $\ell$-candidate. The corresponding indices are $i_l=0<j_l=\oo$ for $l<\ell$, $i_l=i$ for $\ell\leq l<k$ and $i_k=\oo$ for $k\leq l<r$. Condition (2) is vacuously satisfied. These particular canonical $\ell$-candidates will have an important role later. 
\medskip

\textbf{Step 3.} Now, note that any 0-candidate is a canonical 0-candidate and they are all pairwise $F$-similar. Our next claim is that for a fixed sequence of $i$'s the $F$-similarity type of a canonical $\ell$-candidate does not depend on the choice of the $j$'s.
			
			\begin{clm} \label{c:stype}Let $\textbf i=\{i_0\leq i_1 \leq \dots \leq i_{r-1}\}$  be a sequence of ordinals $\leq \omega$ and $\bar \alpha =(\alpha^0_{i_0},\alpha^0_{j_0},\dots, \alpha^{\ell-1}_{i_{\ell-1}},\alpha^{\ell-1}_{j_{\ell-1}},\alpha^\ell_{i_\ell},\dots, \alpha^{r-1}_{i_{r-1}})$ and $\bar \alpha' =(\alpha^0_{i_0},\alpha^0_{j'_0},\dots, \alpha^{\ell-1}_{i_{\ell-1}},\alpha^{\ell-1}_{j'_{\ell-1}},\alpha^\ell_{i_\ell},\dots, \alpha^{r-1}_{i_{r-1}})$ be two canonical $\ell$-candidates (for some sequences $(j_l)_{l<\ell}$ and $(j'_l)_{l<\ell}$). Then $\bar \alpha$ and $\bar \alpha'$ are $F$-similar.
			\end{clm}
			\begin{proof}
				Observe first that for a canonical $\ell$-candidate $\bar \alpha$, the similarity type of $F(\bar \alpha)$ is decided by the values of the $2k$-sequences  $$(F(\alpha^0_{i_0})(\delta^k_{i_k}),F(\alpha^0_{j_0})(\delta^k_{i_k}),\dots, F(\alpha^{k-1}_{i_{k-1}})(\delta^k_{i_k}),F(\alpha^{k-1}_{j_{k-1}})(\delta^k_{i_k}))$$ for each $k<\ell$ (see Figure \ref{fig:blocks2}): indeed, by the definition of similarity we only need to consider the values of the reals in $F''\bar \alpha$ at splitting levels and \eqref{it:eval} of Claim \ref{clm:thin1} guarantees that for the splitting at $\delta^k_{i_k}$ we only need to check the values of reals below $F''A_k$. Also, as $\Delta_F(\alpha^{m}_{j_{m}},\alpha^{m}_\oo)=\delta^m_{j_m}>\delta^k_{i_k}$ whenever $m<k$ and $\alpha^{m}_{j_{m}}\neq \alpha^{m}_\oo$, we get that $$(F(\alpha^0_{i_0})(\delta^k_{i_k}),F(\alpha^0_\oo)(\delta^k_{i_k}),\dots, F(\alpha^{k-1}_{i_{k-1}})(\delta^k_{i_k}),F(\alpha^{k-1}_\oo)(\delta^k_{i_k}))=$$
				$$=(F(\alpha^0_{i_0})(\delta^k_{i_k}),F(\alpha^0_{j_0})(\delta^k_{i_k}),\dots, F(\alpha^{k-1}_{i_{k-1}})(\delta^k_{i_k}),F(\alpha^{k-1}_{j_{k-1}})(\delta^k_{i_k})).$$
				
				The first part of the equation does not depend on the sequence $j_0,\dots,j_{l-1}$ so the second part neither. This, and our first observation, finishes the proof of the claim.

			\end{proof}

		\begin{figure}[H]
			\psscalebox{0.7 0.7}
			{
				\begin{pspicture}(0,-2.9973986)(15.54,2.9973986)
				\psline[linecolor=black, linewidth=0.03](1.92,-2.9826014)(14.08,-0.9626013)
				\psline[linecolor=black, linewidth=0.03](3.16,-2.7626014)(1.58,-1.1226013)
				\psline[linecolor=black, linewidth=0.03](6.44,-2.2426014)(5.12,-0.6026013)
				\psdots[linecolor=black, dotsize=0.16](3.74,2.1973987)
				\psdots[linecolor=black, dotsize=0.16](6.12,2.1773987)
				\psline[linecolor=black, linewidth=0.03](5.14,-0.6426013)(6.1,2.1773987)
				\psline[linecolor=black, linewidth=0.03](3.76,2.1773987)(5.22,-0.38260132)
				\psdots[linecolor=black, dotsize=0.16](2.5,2.1973987)
				\psdots[linecolor=black, dotsize=0.16](0.12,2.2173986)
				\psline[linecolor=black, linewidth=0.03](0.12,2.1973987)(1.64,-0.86260134)
				\psline[linecolor=black, linewidth=0.03](1.58,-1.1426014)(2.52,2.2173986)
				\rput[bl](0.0,2.5773988){$\alpha^0_{i_0}$}
				\rput[bl](1.16,2.5573988){$\alpha^0_{j_0}$}
				\rput[bl](2.32,2.5573988){$\alpha^0_\omega$}
				\rput[bl](4.64,2.4973986){$\alpha^1_{j_1}$}
				\rput[bl](3.54,2.5173986){$\alpha^1_{i_1}$}
				\rput[bl](5.94,2.5373986){$\alpha^1_\omega$}
				\rput[bl](0.66,-0.86260134){$\delta^1_{i_1}$}
				\rput[bl](0.2,-0.14260131){$\delta^{l-1}_{i_{l-1}}$}
				\psline[linecolor=black, linewidth=0.03](9.46,-1.7226013)(9.5,-0.38260132)
				\psline[linecolor=black, linewidth=0.03](9.5,-0.40260133)(10.66,2.1973987)
				\psline[linecolor=black, linewidth=0.03](8.26,2.1773987)(9.7,0.09739868)
				\psdots[linecolor=black, dotsize=0.16](8.26,2.1973987)
				\psdots[linecolor=black, dotsize=0.16](10.64,2.1773987)
				\rput[bl](8.88,2.5173986){$\alpha^{l-1}_{j_{l-1}}$}
				\rput[bl](7.76,2.5373986){$\alpha^{l-1}_{i_{l-1}}$}
				\rput[bl](10.46,2.5173986){$\alpha^{l-1}_\omega$}
				\psdots[linecolor=black, dotsize=0.06](7.08,2.1973987)
				\psdots[linecolor=black, dotsize=0.06](7.28,2.1973987)
				\psdots[linecolor=black, dotsize=0.06](7.48,2.1973987)
				\psdots[linecolor=black, dotsize=0.16644828](1.3,2.2173986)
				\psline[linecolor=black, linewidth=0.03](1.340233,2.1771657)(2.0989225,0.5749503)
				\psline[linecolor=black, linewidth=0.03](4.9,2.1373987)(5.64,0.8373987)
				\psdots[linecolor=black, dotsize=0.16](4.86,2.1773987)
				\psline[linecolor=black, linewidth=0.03](14.1,-0.9626013)(14.88,2.2173986)
				\psline[linecolor=black, linewidth=0.03](12.18,-1.3026013)(12.74,2.2373986)
				\psdots[linecolor=black, dotsize=0.16](12.74,2.1973987)
				\psdots[linecolor=black, dotsize=0.16](14.9,2.1773987)
				\psdots[linecolor=black, dotsize=0.16](9.06,2.1973987)
				\psline[linecolor=black, linewidth=0.03](9.1,2.1573987)(10.02,0.81739867)
				\psdots[linecolor=black, dotstyle=square*, dotsize=0.16](1.38,-0.3626013)
				\psdots[linecolor=black, dotstyle=square*, dotsize=0.16](1.82,-0.38260132)
				\psdots[linecolor=black, dotstyle=triangle*, dotsize=0.16](1.18,0.05739868)
				\psdots[linecolor=black, dotstyle=triangle*, dotsize=0.16](1.9,0.03739868)
				\psdots[linecolor=black, dotstyle=triangle*, dotsize=0.16](4.98,0.07739868)
				\psdots[linecolor=black, dotstyle=triangle*, dotsize=0.16](5.42,0.07739868)
				\psline[linecolor=black, linewidth=0.03, linestyle=dashed, dash=0.17638889cm 0.10583334cm, arrowsize=0.05291667cm 2.0,arrowlength=1.4,arrowinset=0.0]{->}(4.76,-0.3626013)(2.12,-0.3626013)
				\psline[linecolor=black, linewidth=0.03, linestyle=dashed, dash=0.17638889cm 0.10583334cm, arrowsize=0.05291667cm 2.0,arrowlength=1.4,arrowinset=0.0]{->}(9.32,0.07739868)(5.96,0.07739868)
				\psdots[linecolor=black, dotsize=0.0](0.4,-0.24260132)
				\psdots[linecolor=black, dotsize=0.0](0.54,-0.38260132)
				\rput[bl](12.5,2.5573988){$\alpha^{k}_{i_{k}}$}
				\rput[bl](14.72,2.5373986){$\alpha^{r-1}_{i_{r-1}}$}
				\end{pspicture}
			}
			\caption{The important positions marked for deciding similarity types.}
			\label{fig:blocks2}
		\end{figure}
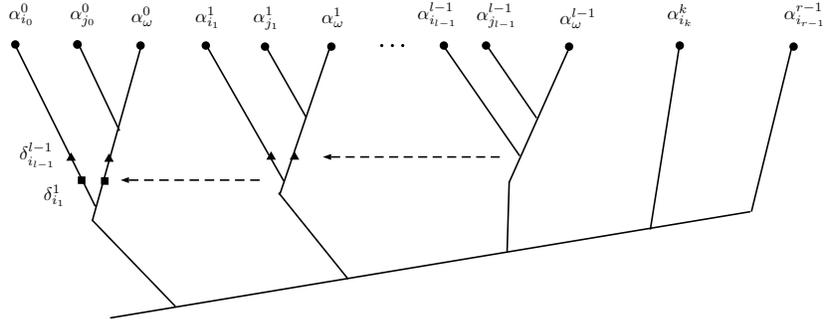


		For a  sequence $\textbf i=\{i_0\leq i_1 \leq \dots \leq i_{r-1}\}$, let us define  $\ev{}{\textbf i}{k}:2k\to 2$  by $$\ev{}{\textbf i}{k}(2m)=F(\alpha^{m}_{i_m})(\delta^{k}_{i_k}) \textmd{ and }\ev{}{\textbf i}{k}(2m+1)=F(\alpha^{m}_\oo)(\delta^{k}_{i_k})$$ for $m<k$ and $i_k\neq \oo$; if $i_k=0$, we declare $\ev{}{\textbf i}{k}=\emptyset$.
		
		If we work with multiple $\textbf A$ sequences at the same time then we mark the above functions as $\ev{\textbf A}{\textbf i}{k}$.
				 With this new notation, in Claim \ref{c:stype} we argued that the $F$-similarity type of a canonical $\ell$-candidate with respect to $\textbf A$ is decided solely by the sequence $(\ev{\textbf A}{\textbf i}{k})_{k<r}$.

		%
		%
		
		\medskip
		
		\textbf{Step 4.} Next, we thin out each $A_l$ further so that the function $\ev{}{\cdot}{k}$, at a fixed $k$, gives the same values for \emph{any choice} of  $\textbf i$.
		
		\begin{clm}\label{clm:thin2}
			There are $\tilde A_l=\{\tilde \alpha^l_i:i\leq \oo\}_<\subseteq A_l$ for $l<r$ so that $\tilde{\textbf A}=(\tilde A_l)_{l<r}$ still satisfies the conditions of Claim \ref{clm:thin1} and there are $\mf e(k):2k\to 2 $ so that for any $\textbf i=\{i_0\leq i_1 \leq \dots \leq i_{r-1} \}$  $$\ev{\tilde {\textbf {A}}}{\textbf i}{k}=\mf e(k).$$
			
		\end{clm}

		\begin{proof}We can do this by a simple application of the classical ($r$-dimensional) Ramsey theorem. Note that we associated with any choice of $\textbf i=\{i_0\leq i_1 \leq \dots \leq i_{r-1} \}\in [\oo]^r$ a sequence $\psi(\textbf i)=(\mf e(\textbf i, k))_{k<r}$. So, we defined $$\psi:[\oo]^r\to\prod_{k<r}(2^{2k}\cup \{\emptyset\}),$$ and hence there is an $I\in [\oo]^\oo$ and $(\mf e(k))_{k<r}$ so that $\psi\uhr [I]^r$ is constant  $(\mf e(k))_{k<r}$.
		
		Now, we simply let  $\tilde A_l=\{\alpha^l_i:i\in I\cup\{\oo\}\}$ for $l<r$. It is easily checked that  the conclusions of Claim \ref{clm:thin1} are still satisfied.

		\end{proof}

		Again, we drop the tilde notation and assume that $\textbf A$ satisfies the conclusions of Claim \ref{clm:thin2}. 

\begin{cor}
        For any $\ell\leq r$, any two canonical $\ell$-candidates with respect to $\textbf A$ are $F$-similar.
\end{cor}	
		

		
		\medskip
		
		\textbf{Step 5.} We are ready to find our infinite monochromatic sumset. Recall that the colourings $c_{s_l}$ are $F$-canonical on $[\bigcup_{i<r} A_i]^{r+l}$ so the colour of a sequence is determined solely by its similarity type. In previous steps, we achieved that for a fixed $\ell\leq r$ all canonical $\ell$-candidates are $F$-similar. In turn, as we have $r+1$ many patterns, we can find $\ell<k\leq r$ so that the map $c_{s_\ell}$ on canonical $\ell$-candidates and the map $c_{s_k}$ on canonical $k$-candidates assumes the same constant value.
		
		
		\begin{clm}
			There is an $X=\{x_i:i<\oo\}\subseteq N(2^{\aleph_0})$ with $\supp(x_i)\subset W$ so that
			\begin{enumerate}
				\item $\supp(2x_i)$ is a canonical $\ell$-candidate and $2x_i=s_\ell*\supp(2x_i)$, and
				\item $\supp(x_i+x_j)$ is a canonical $k$-candidate  and $x_i+x_j=s_k*\supp(x_i+x_j)$ 
			\end{enumerate}for all $i<j<\oo$.
			
		\end{clm}
		In particular,  $c(2x_i)=c(x_i+x_j)$ so $c\uhr X+X$ is constant as desired.
		
		\begin{proof}
			
			We let
			$$\supp(x_i)=\{\alpha^l_0,\alpha^{l}_\oo:l<\ell\}\cup\{\alpha^l_{i}:\ell\leq l<k\}\cup \{\alpha^l_\oo:k\leq l<r\}$$
			
			and define $x_i$ as follows:

			\[
			x_i(\alpha) = \begin{cases}
			1, & \text{for } \alpha\in\{\alpha^l_0,\alpha^{l}_\oo:l<\ell\},\\
			2, & \text{for } \alpha\in \{\alpha^l_i:\ell\leq l<k\}\cup \{\alpha^l_\oo:k\leq l<r\}.
			\end{cases}
			\]
			It is clear that  $\supp(2x_i)=\supp(x_i)$ is a canonical $\ell$-candidate with respect to $\textbf A$. Moreover, $2x_i=s_\ell*\supp(2x_i)$ by the definition of $x_i$.
			
			Now
			$$\supp(x_i+x_j)=\{\alpha^l_0,\alpha^{l}_\oo:l<\ell\}\cup\{\alpha^l_{i},\alpha^l_{j}:\ell\leq l<k\}\cup \{\alpha^l_\oo:k\leq l<r\}$$ for any $i<j<\oo$. This is a canonical $k$-candidate; indeed, we need to check (2) i.e. that $\delta^{l}_{j}>\delta^{k-1}_{i}$ for all $\ell\leq l<k$. This is clear from Claim \ref{clm:thin1} however.
		
		\end{proof}

	\section{Open problems}
	
	There are various problems that remain open at this point. In order to state these questions concisely, we introduce the following notation: given some additive structure $(A,+)$, let $\rd A$ denote the minimal $r$ so that there is an $r$-colouring of $A$ with no monochromatic set of the form $X+X$ for some infinite $X\subseteq A$. Note that the larger $\rd A$ is, the stronger partition property $A$ satisfies.
	
	Now, with this new notation,  we restate Owings' problem: 
	\begin{prob}\cite{owings}
	Does $\rd {\mbb N}=2$ hold?
	\end{prob}
	It was recently noted by J. Zhang\footnote{Personal communication.} that $\rd {\mbb R}=2$ holds in ZFC. However, the next problem remains open:
	
	\begin{prob}
	Does $\rd {\mbb Q}=2$ hold?
	\end{prob}
	
	
	It is easy to see that $\rd{Q(\kappa)}\leq \aleph_0$ for any $\kappa$ and so the conclusion of our main theorem can be rephrased as $\rd {\mbb R}=\aleph_0$. 
	
	 	\begin{prob}
	   Does $\rd {\mbb R}=\aleph_0$ hold if $2^{\aleph_0}$ is real valued measurable?
	\end{prob}

	Also, recall that $h(Q(\aleph_n))\leq 2^{4+n}\cdot 9$ \cite{Hindman}. Now, we ask the following:

	\begin{prob}
	 What is the asymptotic behaviour of the function $n\mapsto \rd{ Q(\aleph_n)}$? Is it truly exponential?
	\end{prob}
	
	A linear lower bound, under GCH, is provided by \cite{leaderrussell}.
	\medskip

	Probably the most intriguing question about large direct sums is the following:
	
	\begin{prob}
Is $\rd{ Q(\aleph_\oo)}=\aleph_0$ provable in ZFC?
	\end{prob}
	
	GCH implies $\rd{ Q(\aleph_\oo)}=\aleph_0$ \cite{leaderrussell} and a positive answer in ZFC would strengthen our main theorem. 

	\medskip
	
The study of Abelian semigroups and finding larger, uncountable monochromatic sumsets  has been already started in \cite{davidassaf,leaderrussell}. It would also be interesting to see what one can say about unbalanced sumsets i.e., sets of the form $X+Y$ with $X,Y\subseteq A$. More precisely:

	\begin{prob} Given an additive structure $(A,+)$ and $r$, characterize those $(\kappa,\lambda)$ such that whenever $c:A\to r$ then there is $X,Y\subseteq A$ with $|X|=\kappa,|Y|=\lambda$ and $c\uhr X+Y$ constant.
	 
	\end{prob}
	
It would be interesting to  analyze similar questions in non Abelian settings as well.

\bibliographystyle{amsplain}


\end{document}